\documentclass[12pt,reqno]{amsart}
\usepackage{amssymb,amsmath,dsfont}
\usepackage{pgfplots}
\pgfplotsset{compat=1.14}
\usepackage{enumitem}
\usepackage{calc}

\title{On the influence of edges in first-passage percolation on $\mathbb{Z}^d$}
\date{}

\author{Barbara Dembin}
\address{Barbara Dembin\hfill\break
    D-MATH, ETH Z\"urich,
    Switzerland.}
\email{barbara.dembin@math.ethz.ch}
\author{Dor Elboim}
\address{Dor Elboim\hfill\break
    School of Mathematics,
    Institute for Advanced Study,
    New Jersey, United States.}
\email{delboim@ias.edu}
\author{Ron Peled}
\address{Ron Peled\hfill\break School of Mathematical Sciences, Tel Aviv University, Tel Aviv, Israel.\hfill\break
School of Mathematics, Institute for Advanced Study and Department of Mathematics, Princeton University, New Jersey, United States.}
\email{peledron@tauex.tau.ac.il}
\usepackage[margin=1in]{geometry}
\usepackage[all]{xy}
\usepackage{amssymb}
\usepackage{amsfonts}
\usepackage{amsthm}
\usepackage{amsmath}
\usepackage{hyperref}
\usepackage{mathtools}
\usepackage{url}
\usepackage{xcolor}
\usepackage{dsfont}
\usepackage{pgfplots}
\usepackage{comment}

\newtheorem{thm}{Theorem}[section]

\newtheorem{lem}[thm]{Lemma}  
\newtheorem{prop}[thm]{Proposition}
\newtheorem{cor}[thm]{Corollary}

\newtheorem{claim}[thm]{Claim}

\newcommand{\ZZ}{\mathbb{Z}}

\mathtoolsset{showonlyrefs}

\numberwithin{equation}{section}

\begin{document}
\begin{abstract}
    We study first-passage percolation on $\ZZ^d$, $d\ge 2$, with independent weights whose common distribution is compactly supported in $(0,\infty)$ with a uniformly-positive density. Given $\epsilon>0$ and $v\in\ZZ^d$, which edges have probability at least $\epsilon$ to lie on the geodesic between the origin and $v$? It is expected that all such edges lie at distance at most some $r(\epsilon)$ from either the origin or $v$, but this remains open in dimensions $d\ge 3$. We establish the closely-related fact that the \emph{number} of such edges is at most some $C(\epsilon)$, uniformly in $v$. In addition, we prove a quantitative bound, allowing $\epsilon$ to tend to zero as $\|v\|$ tends to infinity, showing that there are at most $O\big(\epsilon^{-\frac{2d}{d-1}}(\log \|v\|)^C\big)$ such edges, uniformly in $\epsilon$ and~$v$. The latter result addresses a problem raised by Benjamini--Kalai--Schramm (2003).


    
    Our technique further yields a strengthened version of a lower bound on transversal fluctuations due to Licea--Newman--Piza (1996).
\end{abstract}

\maketitle
\section{Introduction}
First-passage percolation is a model for a random metric space, formed by a random perturbation of an underlying base space. Since its introduction by Hammersley--Welsh in 1965~\cite{HammersleyWelsh}, it has been studied extensively in the probability and statistical physics literature. We refer to \cite{Kesten:StFlour} for general background and to \cite{50years} for more recent results.


We study first-passage percolation on the hypercubic lattice $(\ZZ^d,E(\mathbb{Z}^d))$, $d\ge 2$, in an independent and identically distributed (IID) random environment. The model is specified by a \emph{weight distribution $G$}, which is a probability measure on the non-negative reals. It is defined by assigning each edge $e\in E(\mathbb{Z}^d)$ a random passage time $t_e$ with distribution $G$, independently between edges. Then, each finite path $p$ in $\mathbb Z ^d$ is assigned the random passage time
\begin{equation}\label{eq:time of a path}
    T(p):=\sum _{e\in p} t_e,
\end{equation}
yielding a random metric $T$ on $\ZZ^d$ by setting the \emph{passage time between $u,v\in \mathbb Z ^d$} to
\begin{equation}
    T(u,v):=\inf_{p} T(p),
\end{equation}
where the infimum ranges over all finite paths connecting $u$ and $v$. Any path achieving the infimum is termed a \emph{geodesic} between $u$ and $v$. A unique geodesic exists when $G$ is atomless and will be denoted $\gamma(u,v)$. The focus of first-passage percolation is the study of the large-scale properties of the random metric $T$ and its geodesics.

The passage time of the geodesic between given endpoints is naturally a function of the weights assigned to all edges. To what extent is this passage time \emph{influenced} by the weight assigned to a specific edge? This notion is formalized here by the \emph{probability} that the geodesic passes through that edge. 
It is clear that the influence of edges near the endpoints cannot be uniformly small, but it is not clear whether the influence diminishes uniformly for edges far from the endpoints. This issue was highlighted by Benjamini--Kalai--Schramm~\cite{BKS} in their seminal study of the variance of the passage time, where the following problem, later termed \emph{the BKS midpoint problem}, was posed: Consider the geodesic between $0$ and $v$. Does the probability that it passes at distance $1$ from $v/2$ tend to zero as $\|v\|\to \infty$? The following more general version may also be expected:
\begin{equation}\label{eq:BKS generalization}
\begin{split}
    &\text{show that for any $\epsilon >0$ there is $r(\epsilon)>0$ such that for each $v\in\mathbb{Z}^d\setminus\{0\}$}\\
    &\text{and all $e\in E(\mathbb Z ^d)$ with $d(e,\{0,v\}) > r(\epsilon) $ we have $\mathbb P(e\in\gamma(0,v))< \epsilon$}
\end{split}
\end{equation}
with $d(e,\{0,v\})$ denoting the distance of $e$ from the closer endpoint.

On the square lattice ($d=2$), The BKS midpoint problem was resolved positively by Damron--Hanson~\cite{damron2017bigeodesics} under the assumption that the limit shape boundary is differentiable and then resolved unconditionally by Ahlberg-Hoffman~\cite{ahlberghoffman} (in the more general version~\eqref{eq:BKS generalization}). Recently, assuming that the limit shape has more than $40$ extreme points, the authors \cite{dembin2022coalescence} provided a quantitative version of~\eqref{eq:BKS generalization}, showing that $\mathbb P(e\in\gamma(0,v))$ is smaller than a negative power of $d(e,\{0,v\})$. In all dimensions $d\ge 2$, an optimal, up to sub-power factors, quantitative version of~\eqref{eq:BKS generalization} was obtained by Alexander~\cite{alexander2020geodesics} under assumptions on the model which are still unverified (the proof relies on assumptions on the limit shape and on the passage time fluctuations). Unconditionally, the BKS midpoint problem, and its generalization~\eqref{eq:BKS generalization}), remain open when $d\ge 3$.

Our first main result shows that, in all dimensions, there can be at most a constant number of exceptional edges in~\eqref{eq:BKS generalization}. 
Let us state this precisely. We work with the following class of weight distributions: assume that for some $b>a>0$ and $\beta>0$,
\begin{equation}\label{eq:assumption 1}
\begin{split}
G\text{ is supported on the interval $[a,b]$ and is absolutely continuous}\\
\text{with a density $\rho $ satisfying $\rho (x)\ge \alpha $ for almost all $x\in [a,b]$. } 
\end{split}
\end{equation}
\begin{thm}\label{thm:constant}
    Let $d\ge 2$. Suppose that the weight distribution $G$ satisfies \eqref{eq:assumption 1}. Then, for any $\epsilon >0$ there is a constant $C_\epsilon >0$ depending on $d,G$ and $\epsilon $, such that for all $v\in \mathbb Z ^d$,
    \begin{equation}\label{eq:constant number of edges}
    \big|\big\{ e\in E(\mathbb Z ^d): \mathbb P (e\in \gamma (0,v)) \ge \epsilon \big\} \big| \le C_\epsilon .
\end{equation}
\end{thm}
Note that for small $\epsilon>0$ the left-hand side of~\eqref{eq:constant number of edges} also has a complementary lower bound $c_\epsilon>0$, as there must be edges at a constant distance from the endpoints of the geodesic which have a constant influence.


In the same paper~\cite[page 1975]{BKS}, Benjamini--Kalai--Schramm also posed the following\footnote{The phrasing in \cite[p. 1975]{BKS} differs slightly to fit the setup of Bernoulli weights used there.}:
\begin{equation}\label{eq:BKS question}
\begin{split}
    &\text{show that there exist $C,\delta >0$ such that for each $v\in\mathbb{Z}^d\setminus\{0\}$,}\\
    &\text{$\mathbb P(e\in\gamma(0,v))\le C\|v\|^{-\delta }$ for all but at most $C\|v\|/ \log \|v\|$ edges $e\in E(\mathbb{Z}^d)$}.
\end{split}
\end{equation}
Compared with~\eqref{eq:BKS generalization}, this problem is, on the one hand, more quantitative as the threshold influence depends on $\|v\|$, but is, on the other hand, less restrictive as exceptional edges are allowed. 
A positive answer to this problem would have simplified the proof of the main result of~\cite{BKS}. As such an answer was lacking, the authors of~\cite{BKS} resorted to an averaging trick to circumvent the difficulty. Problem~\eqref{eq:BKS question} and its variants, along with circumventing solutions, also arose in later adaptations of the BKS method~\cite{benaim2008exponential, van2012sublinearity, matic2012sublinear, alexander2013subgaussian, sodin2014positive, Damron2015,dembin2022variance}. Our second main result establishes~\eqref{eq:BKS question} in all dimensions $d\ge 2$. Moreover, a quantitative dependence is obtained, with the bound $\|v\|/\log \|v\|$ of~\eqref{eq:BKS question} improved to $\|v\|^{C\delta }$, which is tight up to the value of the constant $C$ (for small $\delta$). For definiteness, we let $\|v\|$ denote the $\ell ^2$-norm.


\begin{thm}\label{thm:1}
Let $d\ge 2$. Suppose that the weight distribution $G$ satisfies \eqref{eq:assumption 1}. Then, there exists $C>0$, depending only on $d$ and $G$, such that for all $\epsilon >0$ and $v\in \mathbb Z ^d$ with $\|v\|\ge2$, 
\begin{equation}\label{eq:main}
    \big|\big\{ e\in E(\mathbb Z ^d): \mathbb P (e\in \gamma (0,v)) \ge \epsilon \big\} \big| \le C\epsilon ^{-\frac{2d}{d-1}} (\log \|v\|)^{\frac{d(d+1)}{d-1}}.
\end{equation}
\end{thm}
Theorem~\ref{thm:1} provides a power-law dependence on $\epsilon$ in the bound of Theorem~\ref{thm:constant}, at the cost of adding a poly-logarithmic factor in $\|v\|$.



Problem~\eqref{eq:BKS question} was further highlighted in the book~\cite[p. 41 in arXiv version]{50years} where it is pointed out that the only known upper bound on the $\ell^2$ norm of the influences is the trivial
\begin{equation}
    \sum_{e\in E(\mathbb{Z}^d)}\mathbb P(e\in\gamma(0,v))^2\le C\|v\|,
\end{equation}
which follows from the fact that the expected number of edges in $\gamma(0,v)$ is of the order of $\|v\|$. One can use Theorem~\ref{thm:1} together with the latter fact in order to improve this trivial estimate to $\|v\|^{\frac{2}{d+1}}\log(\|v\|)^d$. This is included in the following corollary which bounds general $\ell ^p$-norms.

\begin{cor}\label{cor:main}
Let $d\ge 2$. Suppose that the weight distribution $G$ satisfies \eqref{eq:assumption 1}. Then, for any $\beta \ge1$, there exists $C_\beta >0$, depending only on $d,G$ and $\beta  $, such that for all $v\in \mathbb Z ^d$ with $\|v\|\ge2$, 
\begin{equation}\label{eq:L2 bound on influences}
    \sum _{e\in E(\mathbb Z^d)} \mathbb P ( e\in \gamma (0,v) ) ^\beta  \le C _\beta  \begin{cases}\|v\|^{(\frac{d-1}{d+1})(\frac{2d}{d-1}-\beta)}(\log \|v\| )^{(\beta -1)d} & 1\le \beta < \frac{2d}{d-1} \\ (\log \|v\| )^{\frac{d(d+1)}{d-1}+1} & \beta = \frac{2d}{d-1} \\ (\log \|v\| )^{\frac{d(d+1)}{d-1}} & \beta  >\frac{2d}{d-1}
    \end{cases}.
\end{equation}
\end{cor}

\subsection{Relation with transversal fluctuations lower bound}
To gain a better understanding of the bound~\eqref{eq:main}, it is instructive to have the following picture in mind. It is expected that there exists a dimension-dependent exponent $\xi$, the so-called transversal fluctuation exponent, such that the geodesic $\gamma(0,v)$ typically deviates from the straight line segment connecting $0$ and $v$ by order $\|v\|^\xi$. Moreover, it is expected that if one considers a plane orthogonal to the line segment and intersecting it at distance $\ell$ from the endpoints then the (tight number of) intersection points of the geodesic with the plane lie at distance of order at most $\ell^\xi$ from the line segment, with all points within this range having probability of the same order $\ell^{-\xi(d-1)}$ to be on the geodesic\footnote{While expected, this is not proved - such an estimate would include, as a special case, a quantitative version of the BKS midpoint problem with optimal exponent.} (see Figure~\ref{fig:1}). 
\begin{figure}[htp]
    \centering
    \includegraphics[width=14.5cm]{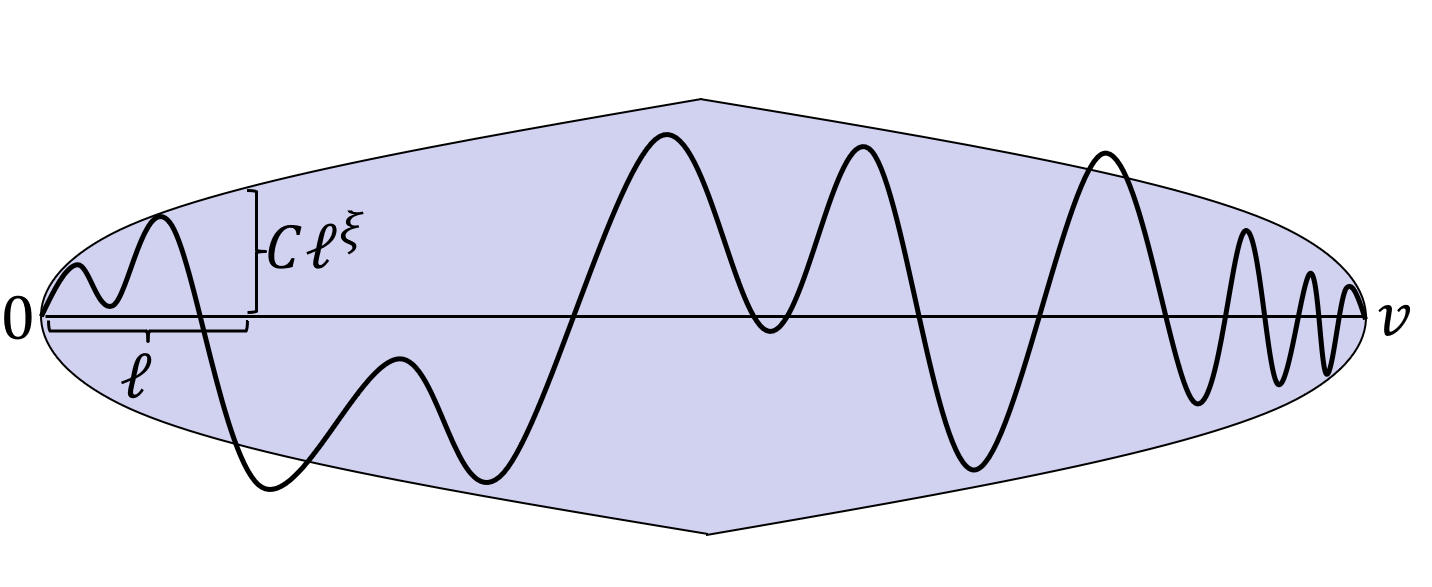}
    \caption{The left half of the colored region illustrates the points in $\mathbb{R}^d$ of the form $\ell\hat{v}+u$, with $\hat{v}=v/\|v\|$, with $\ell\le \frac{1}{2}\|v\|$ and with $u$ perpendicular to $v$ satisfying $\|u\|\le C\ell^\xi$ for some large constant $C$. Here $\xi$ stands for the (putative) transversal fluctuation exponent.  It is expected that edges in the left half incident to some $\ell\hat{v}+u$  have probability of order $\ell^{-\xi(d-1)}$ to be visited by the geodesic. The right half of the colored region has the symmetric properties. Most of the geodesic is expected to be contained in the colored region.}
    
    
    \label{fig:1}
\end{figure}
Consequently, one expects that
\begin{equation}\label{eq:L2 bound on influences for general xi}
    \big|\big\{ e\in E(\mathbb Z ^d): \mathbb P (e\in \gamma (0,v)) \ge \ell ^{-\xi (d-1)} \big\} \big|  \approx  \ell ^{1+\xi (d-1)}
\end{equation}
for all $v\in\mathbb{Z}^d$ and $1\le \ell\le \|v\|$. 
With this in mind, our upper bound~\eqref{eq:main} (up to the poly-logarithmic factor) may thus be regarded as providing the lower bound
\begin{equation}\label{eq:xi lower bound}
\xi\ge \frac{1}{d+1}
\end{equation}
on the transversal fluctuation exponent.

The lower bound~\eqref{eq:xi lower bound} is not expected to be optimal: it is predicted that $\xi=\frac{2}{3}$ for $d=2$ and that $\xi\ge \frac{1}{2}$ for all $d\ge 3$ (see the discussion in~\cite{licea1996superdiffusivity}). Still, a version of~\eqref{eq:xi lower bound} is the best currently known lower bound on transversal fluctuations for a point-to-point geodesic in a fixed direction, established in the work of Licea--Newman--Piza~\cite{licea1996superdiffusivity}. However, as we elaborate next, the version of~\eqref{eq:xi lower bound} proved in~\cite{licea1996superdiffusivity} is too weak to imply~\eqref{eq:main}.

The lower bound on transversal fluctuations stated in~\cite{licea1996superdiffusivity} is
\begin{equation}\label{eq:LiceaNewmanPiza}
    \xi^{(0)}\ge \frac{1}{d+1}
\end{equation}
where 
\begin{equation}
    \xi^{(0)}:=\sup\bigg\{  a\ge 0\colon \lim_{n\to\infty}\sup_{v\in\mathbb{Z}^d,\,\|v\|\ge n}\mathbb{P}\big( \gamma(0,v)\subset\mathrm{cyl}(0,v,\|v\|^a) \big) <1 \bigg\},
\end{equation}
and $\mathrm{cyl}(0,u,r)$ denotes the infinite cylinder of radius $r>0$ around the line connecting $0$ and $u\in\mathbb{R}^d$, defined as
\[\mathrm{cyl}(0,u,r):=\big\{ w\in\mathbb{R}^d: \exists \lambda\in \mathbb R  \,\text{ so that }\|w-\lambda u\| \le r  \big\}.\]
The definition implies that for each positive $a<\xi^{(0)}$, with uniformly positive probability over all but finitely many $v\in\mathbb{Z}^d$, the geodesic $\gamma(0,v)$ contains \emph{at least one point} outside of $\mathrm{cyl}(0,v,\|v\|^a)$. In comparison, the predicted estimate~\eqref{eq:L2 bound on influences for general xi} with $\ell =\|v\|$ yields a stronger conclusion, that for some $c>0$, the expected number of vertices of the geodesic $\gamma(0,v)$ outside of $\mathrm{cyl}(0,v,c\|v\|^{\xi})$ is \emph{at least linear in $\|v\|$}. This stronger conclusion, with $\frac{1}{d+1}$ replacing $\xi$ and a poly-logarithmic correction, is also implied by our bound~\eqref{eq:main}. As it turns out, a direct application of our methods (bypassing~\eqref{eq:main}) allows to remove the poly-logarithmic correction. This is stated precisely in the following theorem, which further provides control also for $\ell<\|v\|$. 
The result is thus a stronger version of the bound~\eqref{eq:LiceaNewmanPiza} of~\cite{licea1996superdiffusivity}.

\begin{thm}\label{thm:2}
Let $d\ge 2$. Suppose that the weight distribution $G$ satisfies \eqref{eq:assumption 1}. Then, there exist $c>0$ and $n_0\ge 1$, depending only on $d$ and $G$, such that for all $v\in \mathbb Z ^d$ with $\|v\|\ge n_0$ and all $(\log \|v\|)^{d+2}\le \ell \le \|v\|$,
\[\mathbb E\left(\big| \big\{ u\in \gamma(0,v) : 0\le (u,\hat{v}) \le \ell \big\} \setminus \mathrm{cyl} \big(0,v,c\ell ^{\frac{1}{d+1}}\big)\big|\right) \ge  \ell /2, \]
where $\hat {v}:=v/\|v\|$ and $(\cdot,\cdot)$ denotes the standard inner product in $\mathbb{R}^d$. In particular (substituting $\ell =\|v\|$), the geodesic $\gamma (0,v)$ contains in expectation at least $\|v\|/2$ vertices outside of $\mathrm{cyl} \big(0,v,c\|v\| ^{\frac{1}{d+1}}\big)$.
\end{thm}
Licea--Newman--Piza~\cite{licea1996superdiffusivity} use martingale methods to prove their results (following Newman--Piza~\cite{newman1995divergence} and Aizenman and Wehr~\cite{aizenman1990rounding,wehr1990fluctuations}) while our proofs rely on perturbing the weights using Lemma~\ref{lem:MW} below. The approaches are different but share a common essence, related to the inequality $\chi\ge\frac{1}{2}(1-(d-1)\xi)$ derived by Wehr--Aizenman~\cite{wehr1990fluctuations} ($\chi$ is the fluctuation exponent of the passage time). After inspecting the proof of the bound~\eqref{eq:LiceaNewmanPiza} in~\cite{licea1996superdiffusivity} we think it may adapt to also yield a version of Theorem~\ref{thm:2} (at least for $\ell=\|v\|$ and possibly under weaker assumptions). However, it is not clear to us whether the method there also adapts to yield versions of Theorem~\ref{thm:constant} and Theorem~\ref{thm:1}.

\subsection{Remarks and open questions}

\begin{enumerate}
\item If the poly-logarithmic factor in the bound of Theorem~\ref{thm:1} was removed then both Theorem~\ref{thm:constant} and Theorem~\ref{thm:2} would follow as corollaries, without the need for separate arguments. However, we did not see a way to do so with our method. 

    \item Our proofs continue to apply under somewhat weaker assumptions than~\eqref{eq:assumption 1}. First, the same proof applies when the assumption that the density $\rho$ is bounded from below is replaced by the assumption that the distribution $G$ is the image of the standard Gaussian distribution under an increasing Lipschitz function from $\mathbb{R}$ to $[a,b]$. Second, with minor modifications to the proof, the left boundary of the support can be chosen to be $a=0$. We also tend to think, but have not established, that a variant of the proof will hold for general absolutely-continuous distributions with sufficiently light tail but this will require an improvement to our Lemma~\ref{lem:MW}.

    \item Of course, it would be significant to resolve the BKS midpoint problem in dimensions $d\ge 3$ or to improve the lower bound~\eqref{eq:xi lower bound} on the transversal fluctuation exponent.
\end{enumerate}

\section{A Mermin--Wagner type estimate}

The following lemma is the main technical tool required for the proof of the main theorems. The lemma is a Mermin--Wagner type estimate and is taken from \cite[Lemma~2.12, Remark~2.15 and Remark~2.16]{dembin2022coalescence}. Let us note that a distribution $G$ satisfying \eqref{eq:assumption 1} is the image of the standard Gaussian distribution under an increasing Lipschitz function and therefore \cite[Remark~2.16]{dembin2022coalescence} holds for such a distribution.

\begin{lem}\label{lem:MW}
  Suppose that $G$ satisfies \eqref{eq:assumption 1}. Then, there exist $C_0>0$ and
  \begin{itemize}
      \item Borel subsets $(B_{\delta })_{\delta>0}$ of $[a,b]$ with $\lim_{\delta\downarrow 0}G(B_\delta)=1$,
      \item For each $\tau \in [0,1]$, an increasing bijection $g^+_{\tau }:[a,b]\to [a,b]$.
  \end{itemize}
   such that the following holds:
  \begin{enumerate}
      \item For any $\tau \in[0,1]$, $w\le g^+_{\tau }(w)\le w+C_0\tau $  for $w\in [a,b] $ and for all $\delta>0$,
      \begin{equation}\label{eq:inclusion implies shift}
          g^+_{\tau }(w)\ge w+\delta \tau \quad\text{for $w\in B_\delta$}.
      \end{equation}
      \item For any $p>1$, an integer $n\ge 1$, a vector $\tau=(\tau_1,\ldots,\tau_n)\in[0,1]^n$ and a Borel set $A\subset\mathbb{R}^n$ we have
      \begin{equation}\label{eq:MW probability estimate}
         \begin{split}
             \mathbb P \Big( \big( g^+_{\tau _1}(X_1),\dots ,g^+_{\tau _n}(X_n) \big) \in A  \Big)\ge \exp \Big(-\frac{p\|\tau \|^2}{2(p-1)} \Big) \cdot \mathbb P \big( (X_1,\dots ,X_n) \in A \big) ^p,
         \end{split} 
      \end{equation}
      where $X_1,X_2,\dots ,X_n$ are i.i.d.\ random variables with distribution $G $. 
  \end{enumerate}
\end{lem}

\section{Proof of Theorem~\ref{thm:constant}}\label{sec:constant}

Fix $\epsilon \le 0.01$. Throughout this section, the constants $C$ and $c$ may depend on $G,d$ and $\epsilon $ and are regarded as generic constants in the sense that their value may change from one appearance to the next. However, constants labeled with a fixed number, such as $C_0$, $c_0$, have a fixed value throughout the paper.
Let $v\in \mathbb Z ^d$ so that $n:=\|v\|$ is sufficiently large and  let $p_e:=\mathbb P (e\in \gamma (0,v))$. Our goal will be to show that the set $A=A_\epsilon :=\{e\in E(\mathbb Z ^d): p_e\ge \epsilon \}$ has constant size. Next, let $h\in \mathbb Z ^d$ and let 
\begin{equation}\label{eq:geodesics}
\gamma _0:=\gamma (0,v), \quad \gamma _h:=\gamma (h,v+h),\quad  T_0:=T(0,v), \quad T_h:=T(h,v+h).
\end{equation}
Let $\tau _e:=|A|^{-1/2}\mathds 1\{e\in A\}$ and let $t_e^+:=g^+_{\tau _e}(t_e)$. We let $T^+(p)$ be the passage time of a path $p$ in the modified environment $(t_e^+)_{e\in E(\mathbb Z ^d)}$.  We also let $\gamma _0^+$ and $\gamma_h^+$ be the corresponding geodesics in the modified environment and let $T_0^+$ and $ T_h^+$ be the corresponding passage times in the modified environment. We turn to show that this change in the environment increases the passage time $T_0$ significantly with positive probability. Recall the sets $B_\delta $ from Lemma~\ref{lem:MW} and that $G(B_\delta ) \to 1$ as $\delta \to 0$. Fix $\delta _0=\delta _0(\epsilon ) >0$ sufficiently small such that for any $k\ge 1$
$$\mathbb P \big( \text{Bin}\big( k,G(B_{\delta _0}) \big) \ge ( 1-\epsilon /4\big) k ) \ge 1-\epsilon ^3.$$
By definition we have that $\mathbb P (\Omega )>1-\epsilon ^3$ where  
\begin{equation}
    \Omega :=\big\{ |\{e\in A : t_e\in B_{\delta _0}\}| \ge (1-\epsilon /4)|A|  \big\}.
\end{equation}

Next, define the event $\mathcal E :=\{|\gamma _0 \cap A|\ge \epsilon |A|/2\}$ and $\mathcal E ^+=\{|\gamma _0 ^+ \cap A|\ge \epsilon |A|/2\}$. We have that  
$$\epsilon |A| \le \sum _{e\in A} p_e =  \mathbb E |\gamma _0 \cap A| \le \mathbb P(\mathcal E ) |A| +\epsilon |A|/2$$
and therefore $\mathbb P (\mathcal E)\ge \epsilon /2$. Thus, By Lemma~\ref{lem:MW} with $p=2$ we have that $\mathbb P (\mathcal E^+) \ge \mathbb P (\mathcal E )^2/e\ge \epsilon ^2/(4e)$ and therefore $\mathbb P (\Omega \cap \mathcal E ^+)\ge \epsilon ^3$.

On the event $\Omega \cap \mathcal E ^+$ we have that $|\{ e\in \gamma _0^+ \cap A :t_e\in B_{\delta _0} \} | \ge \epsilon |A|/4$ and therefore
\begin{equation}\label{eq:21}
\begin{split}
    T_0^+&=\sum _{e\in \gamma _0^+} t_e^+\ge  \sum _{e\in \gamma _0^+} t_e+{\delta _0} \tau _e \mathds 1 \{t_e\in B_{\delta _0} \}\\
    &=  T(\gamma _0^+)+\delta _0 |A|^{-1/2} |\{ e\in \gamma _0^+ \cap A :t_e\in B_{\delta _0} \} |  
    \ge T_0+c \sqrt{|A|}.
\end{split}
\end{equation}
Similarly, we have 
\begin{equation}\label{eq:22}
    T^+_h\le T^+(\gamma_h)=\sum_{e\in \gamma_h} t_e^+\le \sum_{e\in \gamma_h} t_e+C_0\tau _e \le T_h+C|A|^{-1/2}|\gamma _h \cap A|.
    \end{equation}
However, using the triangle inequality and the fact that $t_e\le b$ we have
\begin{equation}\label{eq:23}
T_h\le T(h,0)+ T_0+ T(v,v+h)\le T_0+2b\|h\|_1 .    
\end{equation}
Combining the inequalities \eqref{eq:21},\eqref{eq:22} and \eqref{eq:23} we obtain that on $\Omega \cap \mathcal E ^+$ 
\begin{equation}
    T_h^+\le T_0^++C\frac{|\gamma _h \cap A|}{\sqrt{|A|}} +2b\|h\|_1- c \sqrt{|A|} \le T_h^++C\frac{|\gamma _h \cap A|}{\sqrt{|A|}} +4b\|h\|_1- c \sqrt{|A|},
\end{equation}
where in here we used the triangle inequality once again and the fact that $t_e^+ \le b$. Thus, on  $\Omega \cap \mathcal E ^+$, for any $\|h\| \le c_0 \sqrt{|A|}$ for a sufficiently small $c_0$ we have $|\gamma _h \cap A| \ge c |A|$. Taking expectation and using that $\mathbb P (\Omega \cap \mathcal E ^+)\ge \epsilon ^3$ we have
\begin{equation}\label{eq:77}
   c|A|\le  \mathbb E |\gamma _h\cap A| =\sum _{e\in A} \mathbb P (e\in \gamma _h)= \sum _{e\in A} p_{e-h},
\end{equation}
where the last equality is by translation invariance. Summing this inequality over $h\in \mathbb Z ^d$ with $\|h\| \le c_0\sqrt{|A|}$ and letting $\Lambda _e:=\{e+h : \|h\|\le c_0\sqrt{|A|} \}$ be the ball of radius $c_0\sqrt{|A|}$ around an edge $e$ we obtain 
\begin{equation}
    c|A|^{d/2}|A| \le \sum _{\|h\|\le c_0\sqrt{|A|}} \sum _{e\in A} p_{e-h} =\sum _{e\in A} \mathbb E |\gamma _0 \cap \Lambda _e|  \le C|A|\sqrt{|A|}, 
\end{equation}
where in the last inequality we used that the geodesic between the first entry point of $\gamma $ to $\Lambda _e$ and the last exit point cannot be longer than $C\sqrt{|A|}$ as the weights are supported in $[a,b]$. This finishes the proof of the theorem as $d\ge 2$.

\section{An inequality for smooth functions of the edges}

The main result of this section is the inequality given in the following proposition. In Section~\ref{sec:main theorems} below, we use this inequality in order to prove our main theorems. Let $v\in \mathbb Z ^d$ such that $n:=\|v\|$ is sufficiently large. As before, we let $p_e:=\mathbb P (e\in \gamma (0,v))$.

\begin{prop}\label{cor:1}
Let $q:E(\mathbb Z ^d) \to [0,\infty )$ be a function on the edges that is smooth in the logarithmic scale in the sense that
\begin{equation}\label{eq:ratio}
    0.1\le q_e/q_{e'} \le 10, \quad \text{for any $e,e'$ with $\|e-e'\|\le 2 \log n$},
\end{equation}
where $\|e'-e\|$ denotes the $\ell ^2$ distance between the centers of the edges $e$ and $e'$. Then, there exists a constant $C$ depending only on $G$ and $d$ such that
\begin{equation}\label{eq:inequality}
    \Big( \sum _{e\in E(\mathbb Z^d)} q_e p_e \Big)^d \le C \Big( \sum _{e\in E(\mathbb Z^d)} q_e^2 \Big) ^{\frac{d-1}{2}} \sum _{e\in E(\mathbb Z^d)} q_e.
\end{equation}
\end{prop}

The rest of this section is devoted to the proof of Proposition~\ref{cor:1}. First, we let $\tau _e:=q_e/\|q\|$ where $\|q\|$ is the $\ell ^2$ norm of the vector $(q_e)_{e\in E(\mathbb Z ^d)}$. Note that the inequality in \eqref{eq:inequality} is equivalent to
\begin{equation}\label{eq:inequality'}
    \Big( \sum _{e\in E(\mathbb Z^d)} \tau_e p_e \Big)^d \le C  \sum _{e\in E(\mathbb Z^d)} \tau _e
\end{equation}
and that the function $\tau $ satisfies \eqref{eq:ratio}.

In order to prove \eqref{eq:inequality'} we repeat the same arguments as in Section~\ref{sec:constant} with the function $ |A|^{-1/2}\mathds 1\{e\in A \}$ replaced by the general function $\tau $. That is, we increase the weight of each edge $e$ by $\tau _e$ and analyze the effect it has on the passage times $T_0$ and $T_h$. 

There are two main difficulties in carrying out this argument. The first difficulty, which causes the extra poly-logarithmic factor in Theorem~\ref{thm:1} is that one cannot increase the weights by $\tau _e$ deterministically. Lemma~\ref{lem:MW} allows to increase each weight only with high probability (on the event $\{t_e\in B_\delta \}$). It might be the case that along the geodesic, only edges with small value of $\tau _e$ will be increased and the total change in the passage time will be small. To overcome this issue we use the smoothness assumption in \eqref{eq:ratio} and the percolation argument in Claim~\ref{claim:omega}, showing that on each path of logarithmic length enough edges will be increased.

The second difficulty concerns the random variable $\sum _{e\in \gamma _0^+}\tau _e$ which is the amount in which the passage time $T_0$ is increased and corresponds to $|\gamma _0^+\cap A|$ from Section~\ref{sec:constant}.  The issue is that the typical order of the random variable $\sum _{e\in \gamma _0^+}\tau _e$ can be much smaller than its expectation. That is, the main contribution to the expectation comes from rare events. Taking care of this requires the full strength of Lemma~\ref{lem:MW} to control small probabilities. When the contribution to the expectation comes from rare events in which $\sum _{e\in \gamma _0^+}\tau _e$ is large, the inequality analogous to  \eqref{eq:77} above, deteriorates slightly but this is compensated by the fact that $\|h\|$ can be chosen larger in this case (so that this difficulty does not lead to a loss in the final bound). This trade off can be seen in the statement of Lemma~\ref{lem:1} and in equations \eqref{eq:k} and \eqref{eq:34} below.

Let us move on to the precise argument and start with the following claim. To this end, recall the sets $B_\delta $ from Lemma~\ref{lem:MW} and fix $\delta _0 >0$ sufficiently small such that $G (B_{\delta _0} ) \ge 1-e^{-20d}$. Define the event
\begin{equation}
    \Omega :=\left\{\begin{array}{c}\forall k \ge \log n \text{ and for every path } \Gamma \subseteq [-n^2,n^2]^d \\ \text{of length } k \text{ we have } |\{e\in \Gamma  : t_e\in B_{\delta _0} \}| \ge k/2 \end{array}\right\}.
\end{equation}

\begin{claim}\label{claim:omega}
We have that $\mathbb P (\Omega ) \ge 1-n^{-5}$.
\end{claim}

\begin{proof}
The proof is a simple union bound. For a fixed path $\Gamma $ of length $k$ we have that 
\begin{equation}
 |\{e\in \Gamma  : t_e\in B_{\delta _0} \}|\sim \text{Bin}\big( k,1-e^{-20d}\big)    
\end{equation}
and therefore 
\begin{equation}
    \mathbb P \big(|\{e\in \Gamma  : t_e\in B_{\delta _0} \}|\le k/2\big) \le 2^k e^{-10dk} \le n^{-10-3d} (2d)^{-k} 
\end{equation}
The number of such paths $\Gamma \subseteq [-n^2,n^2]^d$ of length $k$ is at most $(2n^2+1)^d (2d)^k$ and therefore $\mathbb P (\Omega ) \ge 1-n^{-5}$ for all $n$ large enough.
\end{proof}

As in Section~\ref{sec:constant}, we use the notations $\gamma _0,\gamma _h$ and $T_0,T_h$ to denote the geodesics and passage times from $0$ to $v$ and from $h$ to $v+h$ respectively. Next, for a path $p$ define the function $f(p):=\sum _{e\in p } \tau _e $ and let 
\begin{equation}\label{eq:mu}
    \mu :=\mathbb E [f(\gamma _0 )] =\mathbb E \Big[  \sum _{e\in E(\mathbb Z^d)} \tau _e  \mathds 1 \{ e\in \gamma  \} \Big] = \sum _{e\in E(\mathbb Z^d)} \tau _e p_e ,
\end{equation}

\begin{lem}\label{lem:1}
There exists a constant $c_1>0$ such that for all $t>0$ and $\|h\|<c_1t$ we have
\begin{equation}
    \mathbb P \big( f(\gamma _h) \ge c_1t \big) \ge c_1\cdot \mathbb P \big( f(\gamma _0) \ge t \big)^{3/2}-n^{-5}.
\end{equation}
\end{lem}

\begin{proof}[Proof of Lemma~\ref{lem:1}]
Let $\epsilon >0$ sufficiently small and let $\|h\|\le \epsilon t$. Define the weights $t_e^{+}:=g^+_{\tau _e}(t_e)$ for any $e\in E(\mathbb Z ^d)$. Note that $t_e^+=t_e$ for any edge $e$ outside of $[-n^2,n^2]^d$. Indeed, the weight distribution is supported on $[a,b]$ and therefore the geodesic to $v$ cannot travel so far. As in Section~\ref{sec:constant}, we let $T^+(p)$ be the passage time of a path $p$ in the modified environment $(t_e^+)_{e\in E(\mathbb Z ^d)}$. We also let $\gamma _0^+, \gamma _h^+$ and $T_0^+,T_h^+$ be the corresponding geodesics and passage times in the modified environment. Recall the definition of $\delta _0$ and $\Omega $ before Claim~\ref{claim:omega}. We have that 
\begin{equation}\label{eq:T+}
 T_0^+=T^+(\gamma _0^+) =\sum _{e\in \gamma _0^+} t_e^+ \ge T(\gamma _0^+) + \sum _{e\in \gamma _0^+} \mathds 1 \{t_e\in B_{{\delta _0} }\} {\delta _0} \tau _e 
 \ge T_0 + {\delta _0} \sum _{e\in \gamma _0^+} \mathds 1 \{t_e\in B_{{\delta _0} }\} \tau _e .
\end{equation}
In order to estimate the last sum we decompose the path $\gamma _0^+$ into $m:=\lfloor  |\gamma _0^+|/\log n \rfloor $ edge disjoint paths $\Gamma _1,\dots ,\Gamma _m$ such that for all $i\le m$ we have that $\log n \le |\Gamma _i|\le 2\log n$. We also let $e_i$ be the first edge of the path $\Gamma _i$. We have that $t_e^+\in [a,b]$ and therefore $\gamma _0^+ \subseteq [-n^2,n^2]$. Thus, on $\Omega $ for all $i\le m$
\begin{equation}
    \sum _{e\in \Gamma _i} \mathds 1 \{t_e \in B_{\delta _0} \} \tau _e  \ge c \tau_{e_i} \big| \big\{ e\in \Gamma _i : t_e\in B_{\delta _0}  \big\} \big| \ge c  \tau _{e_i} |\Gamma _i| \ge c \sum _{e\in \Gamma _i}  \tau _e,
\end{equation}
where in the first and last inequalities we used \eqref{eq:ratio} and the fact that any edge $e\in \Gamma _i$ satisfies $\|e-e_i\|\le 2\log n$. Substituting this into \eqref{eq:T+}  we get that on the event $\{ f(\gamma _0^+)\ge t \}\cap \Omega $
\begin{equation}\label{eq:61}
 T_0^+\ge T_0 + c \sum _{e\in \gamma _0^+}  \tau _e  \ge  T_0+c t.  
\end{equation}
Moreover, using that $t_e \le b$ and the triangle inequality we obtain 
\begin{equation}\label{eq:62}
    T_0 \ge T_h -2b\|h\|_1 \ge   T_h-C\epsilon  t.
\end{equation}
Similarly, using that $t_e^+\le b$ we have
\begin{equation}\label{eq:63}
   T_0^+\le T_h^++2b\|h\|_1 \le T_h^++C\epsilon  t.
\end{equation}

Combining \eqref{eq:61}, \eqref{eq:62} and \eqref{eq:63} we obtain that on the event  $\{ f(\gamma _0^+)\ge ct \}\cap \Omega $ we have $T_h^+ \ge T_h+ct$ as long as $\epsilon $ is sufficiently small. On the other hand  
\[T_h^+\le T^+(\gamma _h)\le  T_h +\sum_{e\in\gamma_h}C_0\tau _e= T_h+C_0f(\gamma_h)\]
and therefore on $\{ f(\gamma _0^+)\ge t \}\cap \Omega $ we have  $f(\gamma _h) \ge  \epsilon t$.

Finally, by Lemma~\ref{lem:MW} with $p=3/2$ we have 
\begin{equation}
    \mathbb P \big( f(\gamma _h) \ge \epsilon t \big) \ge \mathbb P\big( f(\gamma _0^+)\ge t, \  \Omega  \big) \ge \mathbb P\big( f(\gamma _0^+)\ge t \big)-n^{-5} \ge c\cdot \mathbb P\big( f(\gamma _0)\ge t \big) ^{3/2}-n^{-5}.
\end{equation}
This finishes the proof of the lemma as $c_1$ can be chosen sufficiently small.
\end{proof}

We can now prove Proposition~\ref{cor:1}. 

\begin{proof}[Proof of Proposition~\ref{cor:1}]
Recall the definition of $\mu $ in \eqref{eq:mu}. First, we may assume that $\mu \ge 1$ as otherwise \eqref{eq:inequality'} trivially holds. Indeed, since $\|\tau \|=1$, we have that $\tau _e\le 1$ and therefore $\sum \tau _e\ge \sum \tau _e^2=1$.

We claim that there exists $ -1\le k \le  \log n$ such that 
\begin{equation}\label{eq:k}
    \mathbb P \big( f(\gamma _0) \ge 3^{k} \mu  \big) \ge 4^{-k-3}.
\end{equation}
Indeed, otherwise 
\begin{equation}
\begin{split}
    \mu =\mathbb E [f(\gamma _0)]&\le \mathbb E \big[ f(\gamma _0)\mathds 1 \{ f(\gamma _0)\le \mu /3 \} \big] +\sum _{k=-1}^{\log n} \mathbb E [f(\gamma _0) \mathds 1 \{  3^{k}\mu  \le f(\gamma _0) \le 3^{k+1}\mu   \}]  \\
    & \le \mu /3+ \sum _{k=-1}^{\log n} 3^{k+1}\mu \cdot \mathbb P \big( f(\gamma _0)\ge 3^{k}\mu  \big)  \le  \mu /3
    + 4^{-2}\mu \sum _{k=0}^{\infty } (3/4)^{k} <\mu,     
\end{split}
\end{equation}
where in the first inequality we used that $f(\gamma _0) \le |\gamma _0|\le Cn$ almost surely and that $\mu \ge 1$. Fix $ -1\le k \le  \log n$ such that \eqref{eq:k} holds. By Lemma~\ref{lem:1} with $t=3^k\mu $, for all $\|h\|\le c_13^{k}\mu $ we have 
\begin{equation}\label{eq:34}
\begin{split}
   \sum _{e\in E(\mathbb Z ^d)}& \tau _e p_{e-h} =\mathbb E \bigg[ \sum _{e\in E(\mathbb Z ^d)}\tau _e  \mathds 1\{e\in \gamma _h\} \bigg]=\mathbb E [f(\gamma _h)] \ge c_13^{k}\mu  \cdot \mathbb P \big( f(\gamma _h) \ge c_13^{k}\mu  \big)  \\
    &\ge c_13^{k}\mu  \cdot \big( c\cdot \mathbb P \big( f(\gamma _0 ) \ge 3^k\mu  \big)^{3/2}-n^{-5} \big)  \ge  c(3/8)^{k} \mu=c(3/8)^{k}\sum _{e\in E(\mathbb Z ^d)} \tau _e p_{e},
\end{split}
\end{equation}
where in the last inequality we used \eqref{eq:k} and that $k\le \log n$. Letting $\Lambda _e:=\{e+h:\|h\|\le r\}$ be the ball of radius $r:=c_13^{k}\mu $ around the edge $e$ and summing the last inequality over $h\in \mathbb Z ^d$ with $\|h\|\le r$ we obtain 
\begin{equation}
     cr ^d  (3/8)^{k}  \sum _{e\in E(\mathbb Z^d)} \tau _e p_e \le \sum _{\|h\| \le r} \sum _{e\in E(\mathbb Z ^d)} \tau _e p_{e-h} = \sum _{e\in E(\mathbb Z ^d)} \tau _e  \cdot  \mathbb E [|\gamma \cap \Lambda _e|] \le Cr  \sum _{e\in E(\mathbb Z^d)} \tau _e ,
\end{equation}
where in the last inequality we used that the geodesic between the first entry point of $\gamma $ to $\Lambda _e$ and the last exit point cannot be longer than $Cr$ as the weights are supported in $[a,b]$. Rearranging and substituting $r=c_13^k\mu $ we obtain
\begin{equation}
    \Big( \sum _{e\in E(\mathbb Z^d)} \tau _e  p_e \Big)^d  \le C (8/3^d)^k \sum _{e\in E(\mathbb Z^d)} \tau _e \le C\sum _{e\in E(\mathbb Z^d)} \tau _e ,
\end{equation}
where in here we used that $d\ge 2$. This finishes the proof of the proposition.
\end{proof}

\section{Applications of the inequality}\label{sec:main theorems}

In order to prove Theorem~\ref{thm:1}, we would like to use Proposition~\ref{cor:1} with the function $q_e:=\mathds 1\{p_e>\epsilon \}$. However, this function is not smooth in the sense of \eqref{eq:ratio}. To overcome this issue, we consider a smooth version of this function.

\begin{proof}[Proof of Theorem~\ref{thm:1}]
   Let $A_\epsilon :=\{e\in E(\mathbb Z ^d) : p_e \ge \epsilon \}$ and define the function 
   \begin{equation}\label{eq:defq}
    q_e:=\max\left\{ \exp \bigg( -\frac{\|e'-e\|}{\log n} \bigg)  : \  e'\in A_\epsilon  \right\}.
\end{equation}
It is easy to check that the function $q$ satisfies \eqref{eq:ratio}. Thus, by Proposition~\ref{cor:1} we have 
\begin{equation}
    ( \epsilon |A_\epsilon | )^d \le \Big( \sum _{e\in A_\epsilon }  p_e \Big)^d \le \Big( \sum _{e\in E(\mathbb Z^d) } q_e p_e \Big)^d \le C \Big( \sum _{e\in E(\mathbb Z^d)} q_e^2 \Big) ^{\frac{d-1}{2}} \sum _{e\in E(\mathbb Z^d)} q_e \le C \Big( \sum _{e\in E(\mathbb Z^d)} q_e \Big) ^{\frac{d+1}{2}},
\end{equation}
where in the last inequality we used that $q_e \le 1$. Moreover, we have that 
\begin{equation}
\begin{split}
    \sum _{e\in E(\mathbb Z^d)} \!\! q_e  \le \sum _{e\in E(\mathbb Z ^d)} \sum _{e'\in A_\epsilon }  \! \exp \bigg( -\frac{\|e'-e\|}{\log n} \bigg)=\sum _{e'\in A_\epsilon }  \sum _{e\in E(\mathbb Z ^d)}  \!\! \exp \bigg( -\frac{\|e'-e\|}{\log n} \bigg)  \le C (\log n)^d |A_\epsilon | 
\end{split}
\end{equation}
and therefore we obtain the inequality $(\epsilon |A_\epsilon |)^d \le C (\log n)^{\frac{d(d+1)}{2}} |A_\epsilon | ^{\frac{d+1}{2}}$. Rearranging we get that $|A_\epsilon | \le C(\log n ) ^{\frac{d(d+1)}{d-1}} \epsilon ^{-\frac{2d}{d-1}}$, as needed.
\end{proof}

We turn to prove Theorem~\ref{thm:2}.

\begin{proof}[Proof of Theorem~\ref{thm:2}]
    The proof is similar to the proof of Theorem~\ref{thm:1}. For any $\ell ,r>0$ define $E(\ell ,r)$ be the set of edges in $ E (\mathbb Z ^d)$ with at least one endpoint in $$\text{cyl}(0,v,r) \cap \{x\in \mathbb R ^d: 0\le (x,\hat{v}) \le \ell \}.$$ 
    Define the function 
\begin{equation}\label{eq:defq2}
    q_e:=\max\left\{ \exp \bigg( -\frac{\|e'-e\|}{\log n} \bigg)  : \  e'\in E(\ell ,r)  \right\}
\end{equation}
and note that $q$ satisfies \eqref{eq:ratio}. By Proposition~\ref{cor:1} we have that 
\begin{equation}
    \big( \mathbb E |\gamma (0,v) \cap E(\ell ,r) | \big)^d =\Big( \sum _{e\in E(\ell ,r)} p_e \Big)^d \le \Big( \sum _{e\in E(\mathbb Z^d) } q_e p_e \Big)^d \le C \Big( \sum _{e\in E(\mathbb Z^d)} q_e \Big) ^{\frac{d+1}{2}}.
\end{equation}
Moreover, it is easy to check that if $\ell ,r\ge \log n$ then 
$\sum _{e\in E(\mathbb Z^d)}  q_e \le Cr ^{d-1} \ell $ and therefore 
\begin{equation}
    \mathbb E |\gamma (0,v) \cap E(\ell ,r)|  \le Cr^{\frac{d^2-1}{2d}}\ell ^{\frac{d+1}{2d}}.
\end{equation}
Thus, as long as $n$ is sufficiently large and for all $\ell \ge (\log n)^{d+2}$ we have that 
\begin{equation}
    \mathbb E |\gamma (0,v) \cap E(\ell ,c\ell ^{\frac{1}{d+1}})|  \le \ell /2,
\end{equation}
where the constant $c$ is sufficiently small. This finishes the proof of the theorem as the geodesic $\gamma (0,v)$ contains at least $\ell $ edges that intersect the set $\{x\in \mathbb R ^d : 0\le (x,\hat{v}) \le \ell \}$.
\end{proof}

It remains to prove Corollary~\ref{cor:main}.

\begin{proof}[Proof of Corollary~\ref{cor:main}]
    Since the weights are supported on $[a,b]$ we have that $|\gamma (0,v)|\le Cn$ almost surely and therefore 
     \begin{equation}
        \sum _{e\in E(\mathbb Z^d )}p_e= \mathbb E |\gamma (0,v)| \le Cn. 
    \end{equation}
    We obtain that for all $\epsilon >0$,
    \begin{equation}\label{eq:843}
        \big| \big\{ e\in E(\mathbb Z ^d) :p_e\ge \epsilon \big\} \big| \le C n\epsilon ^{-1}.
    \end{equation}
    Next, letting $j_0:=\lfloor \frac{d-1}{d+1}\log _2 n- d\log _2 \log n \rfloor $, we have for all $\beta >1$
    \begin{equation}\label{eq:839}
    \begin{split}
    \sum _{e\in E (\mathbb Z ^d) } p_e^\beta &\le \sum _{j=1}^{\infty } 2^{-(j-1)\beta } \big| \big\{ e\in E(\mathbb Z ^d) :p_e\in [2^{-j}, 2^{-j+1}] \big\} \big| \\
    &\le C(\log n) ^{\frac{d(d+1)}{d-1}}\sum _{j=1}^{j_0} 2^{j(\frac{2d}{d-1}-\beta ) } + Cn \sum _{j=j_0+1}^{\infty }2^{j(1-\beta)}  \\
    &\le C(\log n) ^{\frac{d(d+1)}{d-1}}\sum _{j=1}^{j_0} 2^{j(\frac{2d}{d-1}-\beta ) } + C_\beta (\log n)^{(\beta -1 )d}n^{\frac{2d-\beta d+\beta }{d+1}},
    \end{split}
    \end{equation}
    where in the second inequality we used Theorem~\ref{thm:1} in the case $1\le j\le j_0$ and \eqref{eq:843} when $j>j_0$. Note that 
    \begin{equation}
        \sum _{j=1}^{j_0} 2^{j(\frac{2d}{d-1}-\beta ) }\le \begin{cases}
            C2^{j_0(\frac{2d}{d-1}-\beta )} \le C (\log n)^{d(\beta -\frac{2d}{d-1})} n^{\frac{2d-\beta d+\beta }{d+1}} \quad &1<\beta \le \frac{2d}{d-1} \\
            Cj_0 \le C\log n \quad &\beta =\frac{2d}{d-1}\\
            C \quad &\beta >\frac{2d}{d-1}
        \end{cases}.
    \end{equation}
    Substituting these estimates into \eqref{eq:839} completes the proof of the corollary. 
\end{proof}

\subsection*{Acknowledgements} We thank Gady Kozma for helping us understand the optimal consequences of inequality \eqref{eq:inequality}. We thank Noga Alon, Michal Bassan, Itai Benjamini, Asaf Nachmias and Allan Sly for fruitful discussions. The research of B.D. is partially funded by the SNF Grant 175505 and the ERC Starting Grant CriSP (grant agreement No 851565) and is part of NCCR SwissMAP. The research of R.P. is supported by the Israel Science Foundation grant
1971/19 and by the European Research Council Consolidator grant 101002733 (Transitions).

Part of this work was completed while R.P. was a Cynthia and Robert Hillas Founders' Circle Member of the Institute for Advanced Study and a visiting fellow at the Mathematics Department of Princeton University. R.P. is grateful for their support.

\bibliographystyle{plain}
\bibliography{biblio}
\end{document}